\RequirePackage{fix-cm}
\documentclass[smallextended,referee,envcountsect,nospthms]{svjour3}   % onecolumn (second format)
\smartqed \usepackage{amsmath, amssymb, latexsym,dsfont}
\usepackage{amsthm}
\usepackage{graphicx}
\usepackage{mathptmx}       % selects Times Roman as basic font
\usepackage{helvet}         % selects Helvetica as sans-serif font
\usepackage{courier}        % selects Courier as typewriter font
\usepackage{type1cm}        % activate if the above 3 fonts are
\usepackage{makeidx}         % allows index generation
\usepackage{subcaption}
\usepackage{subeqnarray}
\usepackage{makecell}                             % when including figure files
\usepackage{epstopdf}
\newtheorem{thr}{Theorem}[section]
\newtheorem{co}[thr]{Corollary}
\newtheorem{lm}[thr]{Lemma}
\newtheorem{pr}[thr]{Proposition}
\newtheorem{ex}[thr]{Example}
\newtheorem{defn}[thr]{Definition}
\newtheorem{rem}[thr]{Remark}

\begin{document}
\title{Dual and Generalized Dual Cones in Banach Spaces}
\titlerunning{Dual and Generalized Dual Cones in Banach Spaces}        % if too long for running head
\author{Akhtar A. Khan, Dezhou Kong, Jinlu Li}
\institute{ Akhtar Khan (Corresponding author)  \at
              School of Mathematical Sciences, Rochester Institute of Technology, Rochester, New York, 14623, USA. \email{aaksma@rit.edu}
             \and
              Dezhou Kong \at
              College of Information Science and Engineering, Shandong Agricultural University, Taian, Shandong,  China.  \email{dezhoukong@163.com}
              \and 
             Jinlu Li \at
              Department of Mathematics, Shawnee State University, Portsmouth, Ohio 45662, USA.  \email{jli@shawnee.edu)}
{\large}}
\date{Received: date / Accepted: date}

\maketitle

\begin{abstract}
This paper proposes and analyzes the notion of dual cones associated with the metric projection and generalized projection in Banach spaces. We show that the dual cones, related to the metric projection and generalized metric projection, lose many important properties in transitioning from Hilbert spaces to Banach spaces. We also propose and analyze the notions of faces and visions in Banach spaces and relate them to metric projection and generalized projection. We provide many illustrative examples to give insight into the given results.
\end{abstract}
\section{Introduction}\label{KL4-S1-I}
Dual cones, induced by the metric projections, have a simple structure and valuable properties in the setting of Hilbert spaces. The derivations of such properties heavily exploit the underlying Hilbertian structure. The Hilbertian structure also equips the metric projection with attractive features, see Zarantonello~\cite[Lemma~1.5]{Zar71}. However, during the last three decades, many important studies of metric projection have been conducted in Banach spaces. This development is partly motivated by the real-world applications of metric projection in optimization, approximation theory, inverse problems, variational inequalities, image processing, neural networks, machine learning, and others. For an overview of these details and some of the related developments, see \cite{BalGol12,BalMarTei21,Bau03,BorDruChe17,Bou15,Bro13,BroDeu72,Bui02,Bur21,CheGol59,ChiLi05,Den01,DeuLam80,DutShuTho17,GJKS21,Ind14,FitPhe82,KonLiuLiWu22,KroPin13,Li04,Li04a,LiZhaMa08,Nak22,Osh70,Pen05,PenRat98,QiuWan22,Ric16,Sha16,ShaZha17,ZhaZhoLiu19},  and the cited references.

The primary objective of this research is to propose and analyze the notion of dual cones associated with the metric projection in Banach spaces. We note that the shortcomings of the metric projection in Banach spaces have resulted in important extensions, namely, the generalized projection and the generalized metric projection, which enjoy better properties in a Banach space framework, see \cite{Alb93,Alb96,KhaLiRwi22,Li04a,Li05}. We show that the dual cones, related to the metric projection and generalized metric projection, lose many properties in transitioning from Hilbert spaces to Banach spaces. We also propose and analyze the notions of faces and visions in Banach spaces and relate them to the metric projection and generalized projection. Illustrative examples are given.

The contents of this paper are organized into five sections. After a brief introduction in Section~1, we recall some background material Section 2. Section 3 studies dual cones related to the metric projection where as the dual cones related to the generalized projection are studied in Section 4. various notions of projections and give new results concerning normalized duality mapping. Section 5 studies the faces and visions in Banach spaces. 
\section{Preliminaries}\label{KL4-S2-P}
Let $X$ be a real Banach space with norm $\|\cdot\|_X$, let $X^*$ be the topological dual of $X$ with norm $\|\cdot\|_{X^*}$, and let $\langle \cdot,\cdot\rangle$ be the duality pairing between $X^*$ and $X$. We will denote the null elements in $X$ and $X^*$ by $\theta$ and $\theta^*.$ Moreover, the closed and convex hull of a set $M\in X$ is denoted by $\overline{\text{co(M)}}.$ Given a Banach space $X$, for $r>0$, we denote the closed ball, open ball and sphere with radius $r$ and center $\theta$ by
\begin{align*}
B(r)&=\{x\in X:\ \|x\|_X\leq r\},\\
B^0(r)&=\{x\in X:\ \|x\|_X<r\},\\
S(r)&=\{x\in X:\ \|x\|_X=r\}.
\end{align*}
For details on the notions recalled in this section, see \cite{Tak00}.

Given a uniformly convex and uniformly smooth Banach space $X$ with dual space $X^*$, the normalized duality map $J:X\to X^*$ is a single-valued mapping defined by
$$\langle Jx,x\rangle=\|Jx\|_{\text{\tiny{\emph{X}}}^*}\|x\|_{\text{\tiny{\emph{X}}}}=\|x\|_{\text{\tiny{\emph{X}}}}^2=\|Jx\|_{\text{\tiny{\emph{X}}}^*}^2,\quad \text{for any}\ x\in X.$$

In a uniformly convex and uniformly smooth Banach space $X$, the normalized map $J_X:X\to X^*$  is one-to-one, onto, continuous and homogeneous. Furthermore, the normalized duality mapping $J^*:X^*\to X$ is the inverse of $J$, that is $J^*J=I_X$ and $JJ^*=I_{X^*}$, where $I_X$ and $I_X^*$ are the identity maps in $X$ and $X^*$. On the other hand, in a general Banach space $X$ with dual $X^*$, the normalized duality mapping $J:X\to 2^{X^*}$ is a set-valued mapping with nonempty valued. In particular, if $X^*$ is strictly convex, then $J:X\to X^*$ is a single-valued mapping. See \cite{Tak00}.

The following example will be repeatedly used in this work.
\begin{ex}\label{KL4-Ex2.1} Let $X=\mathds{R}^3$ be equipped with the $3$-norm $\|\cdot\|_3$ defined for any $z=(z_1,z_2,z_3)\in X,$ by
$$\|z\|_3=\sqrt[3]{|z_1|^3+|z_2|^3+|z_3|^3}.$$ Then $(X,\|\cdot\|_3)$ is a uniformly convex and uniformly smooth Banach space (and is not a Hilbert space). The dual space of $(X,\|\cdot\|_3)$ is $(X^*,\|\cdot\|_{\frac{3}{2}})$ so  that for any $\psi=(\psi_1,\psi_2,\psi_2)$, we have
$$\|\psi\|_{\frac32}=\left(|\psi_1|^{\frac{3}{2}}+|\psi_2|^{\frac{3}{2}}+|\psi_3|^{\frac{3}{2}}\right)^{\frac23}.$$

The normalized duality mappings $J$ and $J^*$ satisfy the following conditions. For any $z=(z_1,z_2,z_3)\in X$ with $z\ne \theta$, we have
\begin{equation}\label{KL3-S2.2-Ex2.2-E2.1}
Jz=\left(\frac{|z_1|^2\text{sign}(z_1)}{\|z\|_3},\frac{|z_2|^2\text{sign}(z_2)}{\|z\|_3},\frac{|z_3|^2\text{sign}(z_3)}{\|z\|_3}\right).
\end{equation}
Moreover, for any $\psi=(\psi_1,\psi_2,\psi_3)\in X^*$ with $\psi\ne \theta$, we have
\begin{equation}\label{KL3-S2.2-Ex2.2-E2.2}
J^*\psi =\left( \frac{|\psi_1|^{\frac32 -1}\text{sign}(\psi_1)}{\left(\|\psi\|_{\frac32}\right)^{\frac32-2}}, \frac{|\psi_2|^{\frac32 -1}\text{sign}(\psi_2)}{\left(\|\psi\|_{\frac32}\right)^{\frac32-2}},\frac{|\psi_3|^{\frac32 -1}\text{sign}(\psi_3)}{\left(\|\psi\|_{\frac32}\right)^{\frac32-2}}\right).
\end{equation}
\end{ex}

Let $X$ be a uniformly convex and uniformly smooth Banach space and let $C$ be a nonempty, closed, and convex subset of $X$. We define a Lyapunov function $V:X^*\times X\to \mathds{R}$ by the formula:
$$V(\psi,x)=\|\psi\|^2_{X^*}-2\langle \psi,x\rangle+\|x\|_X^2,\quad \text{for any}\ \psi\in X^*,\ x\in X.$$

We shall now recall useful notions of projections in Banach spaces.
\begin{defn}
Let $X$ be a uniformly convex and uniformly smooth Banach space, let $X^*$ be the dual of $X$, and let $C$ be a nonempty, closed, and convex subset of $X$.
\begin{description}
\item[1.] The metric projection $P_C:X\to C$ is a single-valued defined by
$$\|x-P_Cx\|_X\leq \|x-z\|_X,\quad \text{for all}\ z\in C.$$
\item[2.] The generalized projection $\pi_C:X^*\to C$ is a single-value map that satisfies
\begin{equation}\label{KL3-S2.5-GMP-E3.3}
V(\psi,\pi_C\psi)=\inf_{y\in C}V(\psi,y),\quad \text{for any}\ \psi\in X^*.
\end{equation}
\end{description}
\end{defn}

The following result collects some of the basic properties of the metric projection defined above.
\begin{pr}\label{KL3-S2.3-NP-P2.6}  Let $X$ be a uniformly convex and uniformly smooth Banach space and let $C$ be a nonempty, closed, and convex subset of $X$.
\begin{description}
\item[1.] The metric projection $P_C:X\to C$ is a continuous map that enjoys the following variational characterization:
\begin{equation}\label{KL3-S2.3-NP-P2.6-E1}
u=P_C(x)\quad \Leftrightarrow \quad \langle J_X(x-u),u-z\rangle\geq 0,\quad \text{for all}\ z\in C.
\end{equation}
\item[2.] The generalized projection $\pi_C:X^*\to C$ enjoys the following variational characterization: For any $\psi\in X^*$ and $y\in C,$
\begin{equation}\label{KL3-S2.4-GP-E2.5}
y=\pi_C(\psi),\quad \text{if and only if},\quad \langle \psi-J_Xy,y-z\rangle\geq 0,\quad \text{for all}\ z\in C.
\end{equation}
\end{description}
\end{pr}

We will also need the following notions.  Given a Banach space $X$, for any  $u,v\in X$ with $u\ne v$, we write
\begin{description}
\item[($a$)] $[v,u]=\{tv+(1-t)u:\ 0\leq t\leq 1\}.$
\item[($b$)] $[v,u\lceil=\{tv+(1-t)u:\ 0\leq t<\infty\}.$
\item[($c$)] $\rceil u,v \lceil=\{tv+(1-t)u:\ -\infty < t <\infty \}.$
\end{description}
The set $[v,u]$ is a closed segment  with end points $u$ and $v$. The set $[v,u\lceil$ is a closed ray in $X$ with end point $v$ with direction $u-v$, which is a closed convex cone with vertex at $v$ and is a special class of cones in $X$.  The set $\rceil u,v \lceil$ is a line in $X$ passing through points $v$ and $u$.

We conclude this section by recalling the following result (see \cite{KhaLi23}):
\begin{thr}\label{T3.1Ten}  Let $X$ be a uniformly convex and uniformly smooth Banach space and let $C$ a nonempty, closed, and convex subset of $X$. For any $y\in C,$ let $x\in X\backslash C$ be such that $y=P_C{x}$. We define the inverse image of $y$ under the metric projection  $P_C:X\to C$ by
$$P_{C}^{-1}(y)=\{u\in X:\ P_{C}(u)=y\}.$$
Then $P_{C}^{-1}(y)$ is a closed cone with vertex at $y$ in $X$. However, $P_{C}^{-1}(y)$ is not convex, in general.
\end{thr}
\section{Dual Cones for the Metric Projection}\label{KL4-S-DCMP}
A cone $K$ in a vector space is said to pointed if $K$ has vertex at $\theta$, $K\cap (-K)=\{\theta\}$, and $K\ne \theta.$

Let $H$ be a Hilbert space, and let $K$ be a cone in $H$ with vertex at $v$. We define the dual cone of $K$ in $H$ with respect to the metric projection $P_K$ by
$$K^{\perp}=\{x\in H:\ \langle x-v,v-z\rangle\geq 0,\ \text{for all}\ z\in K \}.$$

The dual cone has the following properties in Hilbert spaces (see Zarantonello~\cite{Zar71}):
\begin{description}
\item[(1)] $K^{\perp}$ is a closed and convex cone in $H$ with vertex at $v$.
\item[(2)] $K^{\perp\perp}=\overline{\text{co}K}$.
\item[(3)] If $K$ is a closed and convex cone, then $K^{\perp}$ and $K$ are dual cones of each other.
\item[(4)] If $K$ is a closed, convex and pointed cone, then $P_K$ is positive homogeneous and
$$\langle x,P_Kx\rangle=\|P_Kx\|^2,\ \text{for all}\ x\in H.$$
\end{description}

In this following, we extend the concept of a dual cone from Hilbert spaces to uniformly convex and uniformly smooth Banach spaces and derive their valuable properties. We will show that the properties (3) and (4) given above do not hold, in general, in Banach spaces.

\begin{defn}\label{KKL4-D3.1} Let $X$ be a Banach space, let the dual $X^*$ of $X$ be strictly convex, and let $K$ be a cone in $X$ with vertex $v$. We define the dual cone with respect to the metric projection $P$ by
\begin{equation}\label{KKL4-E3.1}
K_P^{\perp}=\{x\in X:\ \langle J(x-v),v-z\rangle\geq 0,\quad \text{for all}\ z\in K\}.
\end{equation}

The following result shows that $K_P^{\perp}$ is a cone in $X$, and $K_P^{\perp}$ and $K$ have the same vertex $v$.
\end{defn}
\begin{thr}\label{KKL4-T3.2} Let $X$ be a Banach space, let the dual $X^*$ of $X$ be strictly convex, and let $K$ be a cone in $X$ with vertex $v$. Then, the following statements hold:
\begin{description}
\item[($a$)] $K_P^{\perp}$ is a cone with vertex at $\theta$ in $X$.
\item[($b$)] If $X$ is uniformly convex and uniformly smooth, then $K_P^{\perp}$ is closed.
\item[($c$)] If $X$ is uniformly convex and uniformly smooth and $K$ is closed and convex, then $K_P^{\perp}=P_K^{-1}(v).$
\item[($d$)] $K_P^{\perp}$ is not convex.
\item[($e$)] $K\not\subseteq (K_{P}^{\perp})_P^{\perp}.$
\item[($f$)] $K$ and $K_P^{\perp}$ are not dual of each other.
\end{description}
\end{thr}
\begin{proof}
($a$) For an arbitrary $x\in K_P^{\perp}$ with $x\ne v$, and for any $t>0$, by the homogeneity property of the normalized duality mapping $J$, we have
$$\langle J(v+t(x-v)-v),v-z\rangle =\langle J(t(x-v)),v-z\rangle=t\langle J(x-v),v-z\rangle\geq 0,\quad \text{for all}\ z\in K, $$
implying that $v+t(x-v)\in K_P^{\perp},$ for all $t>0.$ Thus, $K_P^{\perp}$ is a cone in $X$ with vertex at $v$.

($b$) Under the additional hypothesis on $X$, $J$ is continuous, which proves that $K_P^{\perp}$ is closed.

($c$) By the basic variational principle of $P_K$, for any given $x\in X,$ we have
\begin{equation}\label{KKL4-E3.2}
P_K(x)=v\quad \Leftrightarrow\quad \langle J(x-v),v-z\rangle\geq 0,\quad \text{for all}\ z\in K.
\end{equation}
Since \eqref{KKL4-E3.2} coincides with \eqref{KKL4-E3.1}, we deduce that $K_P^{\perp}=P_K^{-1}(v).$

(d) We construct a counterexample to show that $K_{P}^{\perp}$ is not convex. Take $X=\mathds{R}_3$ given in Example~\ref{KL4-Ex2.1}. Let $u=(-25.-37,-77)$, and
$K=[\theta,u \lceil =\{\alpha u:\ 0\leq \alpha <\infty\}.$ Take $x=(3,-2,-1)$ and $y=(1,-3,2).$ Then $\|x\|_3=\|y\|_3=\sqrt[3]{36}.$ By using (2.1), we have
$$J(x-\theta)=\left(\frac{9}{\|x\|_3},\frac{-4}{\|x\|_3},\frac{-1}{\|x\|_3}\right)=\left(\frac{9}{\sqrt[3]{36}}, \frac{-4}{\sqrt[3]{36}}, \frac{-1}{\sqrt[3]{36}}\right).$$
Next, we compute
%\begin{equation}\label{KKL4-E3.3}
$$\langle J(x-\theta),\theta -\alpha u\rangle=\left\langle \left(\frac{9}{\sqrt[3]{36}}, \frac{-4}{\sqrt[3]{36}}, \frac{-1}{\sqrt[3]{36}}\right),-\alpha (-25,-37,-77) \right\rangle=0,\ \text{for all}\ \alpha u\in K,\ \alpha \in [0,\infty),$$
which implies that $x\in K_P^{\perp}$. Analogously, $J(y-\theta)=\left(\frac{1}{\sqrt[3]{36}}, \frac{-9}{\sqrt[3]{36}}, \frac{4}{\sqrt[3]{36}}\right),$ and hence
$$\langle J(y-\theta),\theta -\alpha u\rangle=\left\langle \left(\frac{1}{\sqrt[3]{36}}, \frac{-9}{\sqrt[3]{36}}, \frac{4}{\sqrt[3]{36}}\right),-\alpha (-25,-37,-77) \right\rangle=0,\ \text{for all}\ \alpha u\in K,\ \alpha \in [0,\infty),$$
which proves that $y\in K_P^{\perp}.$ For $h=\frac23x+\frac13 y$, we have $h-\theta=h=\left(\frac73,-\frac73,0\right)$, yielding
$$J(h-\theta)=\frac{7\sqrt[3]{4}}{6}(1,-1,0).$$
We now compute
$$\langle J(h-\theta),\theta -\alpha u\rangle=-14\sqrt[3]{4}\alpha <0,\quad \text{for every}\ \alpha u\in K,\ \alpha\in (0,\infty),$$
which proves that $h\ne K_P^{\perp}$. Thus, $K_P^{\perp}$ is not convex.

($e$) Since $K_P^{\perp}$ is a closed cone with vertex at $\theta$, $(K_{P}^{\perp})_P^{\perp}$ is a closed cone with vertex at $\theta$.  We will use the counterexample from (d).  Recall that $u=(-25,-37,-77)$ and $K=[\theta,u\lceil$. We showed that $x=(3,-2,-1)\in K_{P}^{\perp}$. Then,
$$\langle J(\alpha u-\theta),\theta -x\rangle=\alpha \langle Ju,-x\rangle<0,  $$
which implies that $\alpha u\notin(K_{P}^{\perp})_P^{\perp} $, for any $\alpha u\in K$ with $\alpha \in (0,\infty)$, which prove ($e$). Finally, ($f$) follows from (e).
\end{proof} %We define%$$P_K^{-1}(v)=\{x\in X:\ P_Kx=v\}.$$
\begin{pr}\label{KKL4-P3.3} Let $X$ be a uniformly convex and uniformly smooth Banach space and let $K$ be a closed, convex, and pointed cone in $X$. Then $P_K$ is positive homogeneous. In general,
\begin{equation}\label{KKL4-E3.4}
\langle Jw,P_Kw\rangle\ne \|P_Kw\|_X^2,\quad \text{for}\ w\in X.
\end{equation}
\end{pr}
\begin{proof} For any $t>0$, since $K$ is a closed, convex and pointed cone in $X$, for any $x\in X$, we have
$$\langle J(tx-tP_Kx),tP_Kx-z\rangle=t^2\langle J(x-P_Kx),P_Kx-t^{-1}z\rangle\geq 0,\quad \text{for all}\ z\in K,$$
and by appealing to the basic variational property of $P_K$, this implies that $P_K(tx)=tP_Kx.$

We next construct a counterexample to prove \eqref{KKL4-E3.4}. Let $x=\mathds{R}^3$ be as in Example~\ref{KL4-Ex2.1}. Let $u=(-25,-37,-77)$ and $K=[\theta, u\lceil$. We take a point $w=(-28,-35,-76)$. Then $w-u=(3,-2,-1).$ By the proof of Theorem~\ref{KKL4-T3.2}, we have
$$\langle J(w-u),u-\alpha u\rangle=(1-\alpha)\left\langle \left(\frac{9}{\sqrt[3]{36}},\frac{-4}{\sqrt[3]{36}},\frac{-1}{\sqrt[3]{36}}\right),(-25,-37,-77)\right\rangle=0,\ \text{for any}\ \alpha u\in K,\ \alpha\in [0,\infty).$$
By the basic variational principle, we have $P_K(w)=u.$ Next we calculate,
$$\|w\|_X=\sqrt[3]{28^3+35^3+76^3}<\|u\|_X=\sqrt[3]{25^3+37^3+77^3},$$
which implies that
$$|\langle Jw,P_Kw\rangle|=|\langle Jw,u\rangle |\leq \|Jw\|_{X^*}\|u\|_X=\|w\|_X\|u\|_X<\|u\|^2_X,$$
which verifies \eqref{KKL4-E3.4}. The proof is complete.
\end{proof}
\section{Generalized dual cones with respect to the generalized projection $\pi$}
We now study the generalized dual cone of $K$ for the generalized projection $\pi$. We first recall some properties of the inverse image of the vertex of a cone by the generalized projection $\pi_C$ in $X^*$.

Given a uniformly convex and uniformly smooth Banach space $X$ with dual space $X^*$ and a cone $K$ with vertex at $v$, we recall that
\begin{equation}\label{KKL4-E4.1}
\pi_C^{-1}(v)=\{\psi\in X^*:\ \pi_C(\psi)=v\}.
\end{equation}
\begin{thr}\label{KKL4-T4.1} Let $X$ be a uniformly convex and uniformly smooth Banach space with dual $X^*$ and let $K$ be a closed and convex cone in $X$ with vertex at $v$. Then,
\begin{description}
\item[(a)] $\pi_K^{-1}(v)$ is a closed and convex cone  in $X^*$ with vertex at $Jv$.
\item[(b)] $\pi^{-1}_{\pi_K^{-1}(v)}(Jv)=K.$
\end{description}
\end{thr}
\begin{proof} (a) See Theorem~\ref{T3.1Ten}. (b) For a fixed $z\in K$, we have
$$\langle \psi-Jv,v-z\rangle\geq 0,\quad \text{for all}\ \psi\in \pi_K^{-1}(v),$$
which, taking into account the identity $J^*=J^{-1}$, implies that
$$\langle Jv-\psi,z-J^*(Jv)\rangle\geq 0,\quad \text{for all}\ \psi\in \pi_K^{-1}(v).$$
By the basic variational principle for $\pi_{\pi_K^{-1}(v)}$, we obtain that $z\in \pi_{\pi_K^{-1}(v)}^{-1}(Jv),$ for all $z\in K,$ proving
\begin{equation}\label{KKL4-E4.2}
K\subseteq \pi_{\pi_K^{-1}(v)}^{-1}(Jv).
\end{equation}

For the converse, for any $z\in \pi_{\pi_K^{-1}(v)}^{-1}(Jv)$, we have $\pi_{\pi_K^{-1}(v)}(z)=Jv$. Appealing to the variational principle for $\pi_{\pi^{-1}_{K}(v)}$ once again, we have
$$\langle Jv-\psi,z-J^*(Jv)\rangle\geq 0,\quad \text{for all}\ \psi\in \pi_K^{-1}(v),$$
and hence
$$\langle \psi-Jv,v-z\rangle\geq 0,\quad \text{for all}\ \psi\in \pi_K^{-1}(v).$$
We recall that for $\ell\in X^*$, we have
\begin{equation}\label{KKL4-E4.4}
\ell\in \pi_K^{-1}(v)\quad \Leftrightarrow\quad \langle \ell-Jv,v-x\rangle\geq 0,\quad \text{for all}\ x\in K.
\end{equation}
For the given $z\in \pi_{\pi_K^{-1}(v)}^{-1}(Jv)$, let $y=z-P_Kz+v.$ Then, for any $x\in K,$ we define
$$w=x+P_Kz-v=v+(P_Kz-v)+(x-v).$$
Since $K$ is a closed and convex cone with vertex $v$, for all $x\in K,$ we have that $w\in K$. By the variational principle for $P_K$, we have
\begin{equation}\label{KKL4-E4.5}
\langle J(y-v),v-x\rangle=\langle J(z-P_Kz),v-x\rangle=\langle J(z-P_Kz),P_Kz-w\rangle\geq 0,\quad \text{for all}\ x\in K,
\end{equation}
which implies that $P_Ky=v$, that is, $y\in P_K^{-1}(v)$. Now, let
\begin{equation}\label{KKL4-E4.6}
\psi=J(y-v)+Jv.
\end{equation}
Then $\psi\in X^*$. By \eqref{KKL4-E4.5}, we have
\begin{equation}\label{KKL4-E4.7}
\langle \psi-Jv,v-x\rangle=\langle J(y-v),v-x\rangle  \geq 0,\quad \text{for all}\ x\in K.
\end{equation}
By \eqref{KKL4-E4.4} and \eqref{KKL4-E4.7}, we have $\psi\in \pi_K^{-1}(v)$. Since $z\in \pi_{\pi_K^{-1}(v)}^{-1}(Jv)$, by the variational principle, we have
$$\langle Jv-\gamma,z-J^*Jv\rangle\geq 0,\quad \text{for all}\ \gamma\in \pi_{K}^{-1}(v) ,$$
that is,
$$\langle \gamma-Jv,v-z\rangle\geq 0,\quad \text{for all}\ \gamma\in \pi_{K}^{-1}(v) ,$$
which, due to the containment $\psi\in \pi_K^{-1}(v)$, implies that
$$\langle \psi-Jv,v-z\rangle\geq 0.$$
Then, using \eqref{KKL4-E4.6}, we have $\langle J(y-v),v-z\rangle\geq 0.$ Then,
$$0\geq \langle J(y-v),z-v\rangle=\langle J(z-P_K(z),z-P_K(z)+P_Kz-v)\rangle=\|z-P_Kz\|^2+\langle J(z-P_Kz),P_Kz-v\rangle $$
which due to $\langle J(z-P_Kz),P_Kz-v\rangle \geq 0$ implies that $\|z-P_Kz\|^2=0,$ that is, $z=P_Kz\in K.$ Since $z$ is arbitrary in $\pi_{\pi_{K}^{-1}(v)}^{-1}(Jv)$, we obtain $ \pi_{\pi_{K}^{-1}(v)}^{-1}(Jv)\subseteq K$. This, in view of \eqref{KKL4-E4.2}, completes the proof.\end{proof}
\begin{defn} Let $X$ be a uniformly convex and uniformly smooth Banach space with dual $X^*$, and let $K$ be a cone in $X$ with vertex at $v$. We define the generalized dual cone of $K$ in $X^*$ with respect to the generalized projection $\pi$ by
\begin{equation}\label{KKL4-E4.12}
K_{\pi}^{\perp}=\{\psi\in X^*:\ \langle \psi-Jv,v-z\rangle\geq 0,\quad \text{for all}\ z\in K \}.
\end{equation}
\end{defn}
\begin{thr}\label{KKL4-T4.3} Let $X$ be a uniformly convex and uniformly smooth Banach space with dual $X^*$, and let $K$ be a cone in $X$ with vertex at $v$. Then the following statements hold:
\begin{description}
\item[(a)] $K_{\pi}^{\perp}=\pi_{K}^{-1}(v)$.
\item[(b)] $K_{\pi}^{\perp}$ is a closed and convex cone with vertex at $Jv$ in $X^*$.
\item[(c)] $K$ and $K_{\pi}^{\perp}$ are generalized dual of each other: $K=(K_{\pi}^{\perp})_{\pi}^{\perp}.$
\end{description}
\end{thr}
\begin{proof} (a). By the basic variational principle for $\pi_K$, for any $\psi\in X^*$ and $v\in K$, we have
$$v=\pi_K(\psi)\quad \Leftrightarrow\quad \langle \psi-Jv,v-z\rangle\geq 0,\quad \text{for all}\ z\in K,$$
which, due to the definition of $K_{\pi}^{\perp}$ and \eqref{KKL4-E4.1} at once implies (a). (b) Since $\pi_{K}^{-1}(v)$ is a closed and convex cone with vertex at $Jv$ in $X^*$, (b) follows at once (a). Finally, (c) follows from (a) and Theorem~\ref{KKL4-T4.1}.
\end{proof}

\begin{co} Let $X$ be a uniformly convex and uniformly smooth Banach space, and let $C$ and $K$ be closed and convex cones in $X$ with a common vertex at $v$ satisfying $C\cap K\ne \{v\}$. Then,
\begin{description}
\item[(a)] $C\subseteq K \quad \Leftrightarrow\quad C_{\pi}^{\perp}\supseteq K_{\pi}^{\perp}.$
\item[(b)] $(C\cap K)_{\pi}^{\perp}=\overline{\text{co}(C_{\pi}^{\perp}\cup K_{\pi}^{\perp})}.$
\end{description}
\end{co}
\begin{proof} (a) The proof of $C\subseteq K \ \Rightarrow\ C_{\pi}^{\perp}\supseteq K_{\pi}^{\perp}$ is evident. The converse follows from part (c) of Theorem~\ref{KKL4-T4.3}.

(b) It follows at once that $C\cap K$ is a closed and convex cone in $X$ with vertex $v$. By (a), the inclusion $C\cap K\subseteq C$ implies that $(C\cap K)_{\pi}^{\perp}\supseteq C_{\pi}^{\perp}$ and the inclusion $C\cap K\subseteq K$ implies that $(C\cap K)_{\pi}^{\perp}\supseteq K_{\pi}^{\perp}$, and hence $(C\cap K)_{\pi}^{\perp}\supseteq C_{\pi}^{\perp}\cap C_{\pi}^{\perp}.$ However, since $(C\cap K)_{\pi}^{\perp}$ is a closed and convex cone with vertex at $Jv$, it follows that
$$(C\cap K)_{\pi}^{\perp}\supseteq \overline{\text{co}(C_{\pi}^{\perp}\cup K_{\pi}^{\perp})}. $$
By (c) of Theorem~\ref{KKL4-T4.3} and (a), we have
\begin{equation}\label{KKL4-E4.16} C\cap K = ((C\cap K)_{\pi}^{\perp})_{\pi}^{\perp}\subseteq (\overline{\text{co}(C_{\pi}^{\perp}\cup K_{\pi}^{\perp})})_{\pi}^{\perp}.
\end{equation}
On the other hand, from $C_{\pi}^{\perp}\subseteq \overline{\text{co}(C_{\pi}^{\perp}\cup K_{\pi}^{\perp})}$ and $K_{\pi}^{\perp}\subseteq \overline{\text{co}(C_{\pi}^{\perp}\cup K_{\pi}^{\perp})}$, we have
$C=(C_{\pi}^{\perp})_{\pi}^{\perp}\supseteq \overline{\text{co}(C_{\pi}^{\perp}\cup K_{\pi}^{\perp})}$ and $K=(K_{\pi}^{\perp})_{\pi}^{\perp}\supseteq \overline{\text{co}(C_{\pi}^{\perp}\cup K_{\pi}^{\perp})}$. Thus, by \eqref{KKL4-E4.16}, we have
$$C\cap K=((C\cap K)_{\pi}^{\perp})_{\pi}^{\perp}\subseteq (\overline{\text{co}(C_{\pi}^{\perp}\cup K_{\pi}^{\perp})})_{\pi}^{\perp}\subseteq C\cap K,$$
which proves the desired identity. Since $(C\cap K)_{\pi}^{\perp}$ and $\overline{\text{co}(C_{\pi}^{\perp}\cup K_{\pi}^{\perp})}$ are both closed and convex cones with vertex at $Jv$, we have the result by using Theorem~\ref{KKL4-T4.3}.
\end{proof}

The following result can be proved in an analogous fashion:
\begin{co} Let $X$ be a uniformly convex and uniformly smooth Banach space, and let $\{K_{\lambda}:\ \lambda \in \Lambda\}$ be a set of closed and convex cones with a common vertex at $v$ such that $\cap_{\lambda\in \Lambda}K_{\lambda}\ne \{v\}.$
Then
$$(\cap_{\lambda \in \Lambda} (K_{\lambda})_{\pi}^{\perp} =\overline{\text{co}(\cup_{\lambda\in \Lambda}(K_{\lambda})_{\pi}^{\perp})},$$
where $\Lambda$ is an arbitrary given index set.
\end{co}
\section{Faces and visions in Banach spaces}
\subsection{Faces in Banach spaces}
\begin{defn}\label{KKL4-D5.2} Let $X$ be a Banach space with dual $X^*$ and let $C$ be a nonempty, closed, and convex subset of $X$. For any $\psi\in X^*,$ we define the face of $\psi$ on $C$ by
$$\mathcal{F}_C(\psi)=\{y\in C:\ \langle \psi,y\rangle=\sup_{x\in C}\langle \psi,x\rangle \}.$$
\end{defn}

\begin{rem}\label{KKL4-L5.3}
It is evident from the above definition that for any $\psi\in X^*,$ the set $\mathcal{F}_C(\psi)$ is either empty or a closed and convex subset of $C$. Moreover, $\mathcal{F}_C(\theta^*)=C.$
\end{rem}

Before proceeding any further, we gather a few examples to illustrate the above notion.
\begin{ex}\label{KKL4-Ex5.4} Let $X=\mathds{R}^3$ be as in Example~\ref{KL4-Ex2.1}. We take $u=(25,37,77)$ and let $C=[\theta, u\lceil=\{tu\in \mathds{R}^3:\ t\geq 0\}.$
\begin{description}
\item[(a)] Let $\psi=(-9,4,1)\in (\mathds{R}^3)^*$. Then $\mathcal{F}_C(\psi)=C.$
\item[(b)] Let $\psi=(\psi_1,\psi_2,\psi_3)\in (\mathds{R}^3)^*\ \text{with}\ \psi_{i}\leq 0,\ \text{for}\ i=1,2,3\ \text{and}\ \psi_1+\psi_1+\psi_3<0.$ Then $\mathcal{F}_C(\psi)=\{\theta\}.$
\item[(c)] Let $\psi=(\psi_1,\psi_2,\psi_3)\in (\mathds{R}^3)^*\ \text{with}\ \psi_{i}\geq 0,\ \text{for}\ i=1,2,3\ \text{and}\ \psi_1+\psi_1+\psi_3>0.$ Then $\mathcal{F}_C(\psi)=\emptyset.$
\end{description}
\end{ex}
\begin{proof} (a). For $\psi=(-9,4,1)\in (\mathds{R}^3)^*$, we have
$$\langle \psi,tu\rangle=\langle (-9,4,1),tu\rangle=t\langle (-9,4,1),(25,37,77)\rangle=0,\quad \text{for any}\ ut\in X,\ t\geq 0,$$
which implies that $tu\in \mathcal{F}_C(\psi)$ for all $tu\in C$. Parts (b) and (c) can be proved analogously.
\end{proof}
\begin{ex}\label{KKL4-Ex5.5} Let $(S,\mathcal{A},\mu)$ be a measure space with $\mu(S)\geq 1.$ For any $p\in [1,\infty),$ let $X=L_p(S)$ be the real Banach space of real functions defined on $S$ with norm $\|\cdot\|_p$. For any given $M>0$, let
$$C=\{f\in L_p(S):\ \|f\|_p\leq M\}.$$
Then $C$ is a nonempty, closed, and convex subset in $L_p(S)$. For any $A\in \mathcal{A}$ with $1\leq \mu(A)<\infty$, let $1_A$ denote the characteristic function of $A$, which satisfies $1_A\in L_q(S)^*=L_p(S)$, where $p,q\in [1,\infty]$ are such that $\frac{1}{p}+\frac{1}{q}=1.$ Then $\mathcal{F}_C(1_A)$ is a nonempty, closed, and convex subset of $C$ such that
\begin{equation}\label{KKL4-E5.1}
\mathcal{F}_C(1_A)=\left\{g\in C:\ \int_Ag(s)d\mu(s)=M\right\}.
\end{equation}
\end{ex}
\begin{proof} For any $g\in C,$ if $\int_Ag(s)d\mu(s)=M$, then
\begin{equation}\label{KKL4-E5.2}
\langle 1_A,g\rangle=\int_S 1_A(s)g(s)d\mu(s)=\int_Ag(s)d\mu(s)=M.
\end{equation}
For any $f\in C,$ we have
\begin{equation}\label{KKL4-E5.3} \langle 1_A,f\rangle=\int_S 1_A(s)f(s)d\mu(s)=\int_Af (s)d\mu(s)\leq \|f\|_p\leq M.
\end{equation}
Thus, \eqref{KKL4-E5.2} and \eqref{KKL4-E5.3} imply that
\begin{equation}\label{KKL4-E5.4}
\left\{g\in C:\ \int_Ag(s)d\mu(s)=M\right\} \subseteq \mathcal{F}_C(1_A).
\end{equation}

For the converse,  we define $h$ on $S$ by
$$h(s)=\frac{M}{\mu(A)}1_A(s),\quad \text{for all}\ s\in S.$$
By $1\leq \mu(A)<\infty$, we have
$$\|h\|_p=\frac{M}{\mu(A)^{\frac{p-1}{p}}}\leq M,$$
which implies that $h\in C$. By $1_A\in L_p(S)^*$, we have
\begin{equation}\label{KKL4-E5.5}
\langle 1_A,h\rangle=\int_A1_A(s)h(s)d\mu(s)=\int_A\frac{M}{\mu(A)}1_A(s)d\mu(s)=M.
\end{equation}
By the above equation, it follows that for any $g\in \mathcal{F}_C(1_A)\subseteq C$, we must have
$$M\geq \|g\|_p\geq \langle 1_A,g\rangle\geq \langle 1_A,h\rangle=M. $$
It follows that $\langle 1_A,g\rangle=M $, that is, $\int_Ag(s)d\mu(s)=M.$ This implies that
\begin{equation}\label{KKL4-E5.6}
\mathcal{F}_C(1_A)\subset \left\{g\in C:\ \int_Ag(s)d\mu(s)=M\right\}.
\end{equation}
By combining \eqref{KKL4-E5.4} and \eqref{KKL4-E5.6}, we get \eqref{KKL4-E5.1}. By \eqref{KKL4-E5.5} and \eqref{KKL4-E5.4}, we have $h\in \mathcal{F}_C(1_A)$, which shows that $\mathcal{F}_C(1_A)\ne \emptyset.$ This prove the claim.
\end{proof}
\begin{ex}\label{KKL4-Ex5.6} For any $p$ with $1\leq p<\infty$, let $X=\ell_p$ be the real Banach space of real sequences with norm $\|\cdot\|_p$. For any given $M>0$, let 
$$C=\{x\in \ell_P:\ \|x\|_p\leq M\}.$$
Then $C$ is a nonempty, closed and convex subset in $\ell_p.$ For any positive integers $m$ and $n$ with $n\geq 1.$ We define $\ell_m^n\in (\ell_p)^*=\ell_q$ by
$$(\ell_m^n)_i=
\left\{\begin{array}{lll}1,&\text{for}&i=m,m+1,\ldots,m+n-1,\\ 0,&&\text{otherwise}. \end{array}\right.$$
Then, $\mathcal{F}_C(\ell_m^n)$ is a nonempty, closed, and convex subset of $C$ such that 
\begin{equation}\label{KKL4-E5.7}
\mathcal{F}_C(\ell_m^n)=\left\{y=\{y_i\}\in C:\ \sum_{i=m}^{m+n-1}y_i=M \right\}.
\end{equation}
\end{ex}
\begin{proof} We only need to show that $\mathcal{F}_C(\ell_m^n)$ is nonempty. For this, we take $z=\{z_i\}\in \ell_p$ as follows:
$$z_i=
\left\{\begin{array}{lll}\frac{M}{n},&\text{for}&i=m,m+1,\ldots,m+n-1,\\ 0,&&\text{otherwise}. \end{array}\right.$$
Then, it is easy to verify that $z\in C$ and $z\in \mathcal{F}_C(\ell_m^n)$. The rest of the arguments are similar to the ones used in Example~\ref{KKL4-Ex5.5}.
\end{proof}
\begin{lm}\label{KKL4-L5.7} Let $X$ be a reflexive Banach space with dual space $X^*$ and let $C$ be a closed, convex and bounded set in $X$. Then for each $\psi\in X^*$, $\mathcal{F}_C(\psi)$ is nonempty, closed, and convex subset of $C$. \end{lm}
\begin{proof} Since $C$ is weakly compact, for any $\psi\in X^*$, the function $\langle \psi,\cdot\rangle $ attains its maximum value on $C$. That is, there is $y\in C$ such that  $\langle \psi,y\rangle=\max_{x\in C}\langle \psi, x\rangle $. This implies that $y\in \mathcal{F}_C(\psi).$ The set $\mathcal{F}_C(\psi)$ is clearly, closed and convex.
\end{proof}
\begin{thr}\label{KKL4-T5.8} Let $X$ be a reflexive Banach space with dual $X^*$ and let $C$ be a nonempty, closed, and convex set in $X$. Then
\begin{description}
\item[(a)] For any $u\in X,$
$$\mathcal{F}_C(Ju)=\{y\in C:\ y=P_C(u+y)\}=\{y\in C:\ y=\pi_C(Ju+Jv)\}.$$
\item[(b)] For any $\psi\in X^*,$
$$\mathcal{F}_C(\psi)=\{y\in C:\ y=P_C(J^*\psi+y)\}=\{y\in C:\ y=\pi_C(\psi +Jy)\}.$$
\end{description}
\end{thr}
\begin{proof} (a) For an arbitrary $z\in C,$ by the basic variational principle for $P_C,$ we have
\begin{align*}
z\in \mathcal{F}_C(Ju) \quad \Leftrightarrow \quad &0\leq \langle Ju,z-x\rangle,\quad \text{for all}\ x\in C\\ \quad \Leftrightarrow \quad &0\leq \langle J(u+z-z),z-x\rangle,\quad \text{for all}\ x\in C.\\
 \quad \Leftrightarrow \quad &z=P_C(u+z)\\
  \quad \Leftrightarrow \quad &z\in \{y\in C:\ y=P_C(u+y)\}.
\end{align*}
This proves the first equality in (a). To prove the second inequality, for any $z\in C$, by the basic variational principle of $\pi_C$, we have
\begin{align*}
z\in F_C(Ju) \quad \Leftrightarrow \quad &0\leq \langle Ju,z-x\rangle,\quad \text{for all}\ x\in C\\ \quad \Leftrightarrow \quad &0\leq \langle Ju+Jz-Jz,z-x\rangle,\quad \text{for all}\ x\in C.\\
 \quad \Leftrightarrow \quad &z=\pi_C(Ju+Jz)\\
  \quad \Leftrightarrow \quad &z\in \{y\in C:\ y=\pi_C(Ju+Jy)\},
\end{align*}
which proves the second equality in (a).

(b) For any $\psi\in X^*,$ $J^*\psi\in X$ by substituting $J^*\psi$ for $u\in X$ in (a) and noticing $JJ^*\psi=\psi$, (b) follows at once. \end{proof}

The conclusion of Theorem~\ref{KKL4-T5.8} can be described by the form of variational inequalities.
\begin{co}\label{KKL4-C5.9} Let $X$ be a uniformly convex and uniformly smooth Banach space wth dual $X^*$ and let $C$ be a nonempty, closed, and convex set in $X$. Then
\begin{description}
\item[(a)] For any $u\in X,$ a point $y\in C$ is a solution of the variational inequality
$$\langle Ju,y-x\rangle\geq 0,\quad \text{for all}\ x\in C,$$
if and only if, $y$ is a solution to one of the following projection equations:
$$y=P_C(u+y)\quad \text{or}\quad y=\pi_C(Ju+Jy).$$
\item[(b)] For any $\psi\in X^*,$ a point $y\in C$ is a solution of the variational inequality
$$\langle \psi,y-x\rangle\geq 0,\quad \text{for all}\ x\in C,$$
if and only if, $y$ is a solution to one of the following projection equations:
$$y=P_C(J^*\psi+y)\quad \text{or}\quad y=\pi_C(\psi+jy).$$
\end{description}
\end{co}

\subsection{Visions in Banach spaces}
\begin{defn}\label{KKL4-D5.10} Let $X$ be a Banach space with dual $X^*$ and let $C\subset X$ be nonempty, closed, and convex.
\begin{description}
\item[(a)] We define the vision $\mathcal{F}_C^{-1}(y)$ in $X^*$ of a point $y\in C$ with respect to the background $C$ by
$$\mathcal{F}^{-1}_C(y)=\{\psi\in X^*:\ y\in \mathcal{F}_C(\psi)=\{\psi\in X^*: \langle \psi,y\rangle=\sup_{x\in C}\langle \psi,x\rangle \}.$$
\item[(b)] We define the vision $\mathcal{F}_C^{-2}(y)$ in $X$ of a point $y\in C$ with respect to the background $C$ by
$$\mathcal{F}_C^{-2}(y)=\{u\in X:\ y\in \mathcal{F}_C(Ju)\}=\{u\in X:\langle Ju,y\rangle=\sup_{x\in C}\langle Ju,x\rangle  \}$$
\end{description}
\end{defn}

\begin{lm}\label{KKL4-L5.12}  Let $X$ be a uniformly convex and uniformly smooth Banach space with dual $X^*$, and let $C$ be a nonempty, closed, and convex subset in $X$. Then, for any $y\in C,$ we have
$$\mathcal{F}_C^{-2}(y)=J^*(\mathcal{F}_C^{-1}(y))\quad \text{or}\quad \mathcal{F}_C^{-1}(y)=J(\mathcal{F}_C^{-2}(y)).$$
\end{lm}
\begin{proof} Since in a uniformly convex and uniformly smooth Banach space $X$, $J$ and $J^*$ are both one-to-one and onto mapping such that $J^*J=I_X$ and $JJ^*=I_{X^*},$ the conclusions are evident.
\end{proof}

\begin{pr}\label{KKL4-P5.12}  Let $X$ be a Banach space with dual $X^*$ and let $C\subset X$ be nonempty, closed, and convex. Then, for any $y\in C,$ we have
\begin{description}
\item[(a)] $\theta^*\in \mathcal{F}_C^{-1}(y)$ and $\mathcal{F}_C^{-1}(y)\ne \emptyset.$
\item[(b)] If $\{\theta^*\} \subsetneqq \mathcal{F}_C^{-1}(y)$, then $\mathcal{F}_C^{-1}(y)$ is a closed and convex cone with vertex at $\theta^*.$
\end{description}
\end{pr}
\begin{proof}  Since (a) is evident, we only prove (b). For any $\psi\in \mathcal{F}_C^{-1}(y)$ and $t\geq 0$, we have
$$\langle t\psi,y\rangle=t\sup_{x\in C}\langle \psi,x \rangle=\sup_{x\in C}\langle t\psi,x\rangle,$$
which  implies that $t\psi\in \mathcal{F}_C^{-1}(y)$, and hence $\mathcal{F}_C^{-1}(Y)$ is a cone with vertex at $\theta^*$ in $X^*.$

For any $\psi,\phi\in \mathcal{F}_C^{-1}(y)$ and for any $t \in [0,1]$ by $y\in C$, we have
\begin{align*}
\sup_{x\in C}\langle t\psi+(1-t)\phi ,x\rangle&\geq \langle t \psi+(1-t)\phi,y\rangle =t \langle \psi,y\rangle+(1-t)\langle \phi,y\rangle=t \sup_{x\in C}\langle \psi,x\rangle+(1-t)\sup_{x\in C}\langle \phi,x\rangle\\
& = \sup_{x\in C}\langle t\psi,x\rangle+\sup_{x\in C}\langle (1-t)\phi,x\rangle\geq \sup_{x\in C}\langle t\psi+(1-t)\phi,x\rangle,
\end{align*}
which implies that
$$\langle t\psi+(1-t)\phi,y\rangle=\sup_{x\in C}\langle t\psi+(1-t)\phi,x\rangle$$
and hence $t\psi+(1-t)\phi\in \mathcal{F}_C^{-1}(y)$, proving the desired convexity. To prove that $\mathcal{F}_C^{-1}(y)$ is closed in $X^*$. Let $\{\psi_n\}\subseteq \mathcal{F}_C^{-1}(y)$ and $\psi\in X^*$ be such that $\psi_n\to \psi$ in $X^*$ as $n\to \infty. $ This implies that
$$\langle \psi,y\rangle=\lim_{n\to \infty}\langle \psi_n,y\rangle\geq \lim_{n\to \infty}\langle \psi_n,x\rangle=\langle \psi,x\rangle,  $$
which proves that  $\psi \in \mathcal{F}_C^{-1}(y)$, and hence $\mathcal{F}_C^{-1}(y)$ is closed in $X^*$.
\end{proof}

\begin{pr}\label{KKL4-P5.13}  Let $X$ be a Banach space with dual space $X^*$ and let  $C$ be a nonempty, closed, and convex subset in $X$. Then for any $y\in C$, we have
\begin{description}
\item[(a)] $\theta in \mathcal{F}_C^{-2}(y)$ and $\mathcal{F}_C^{-2}(y)\ne \emptyset.$
\item[(b)] If $\mathcal{F}_C^{-2}\varsupsetneq \{\theta\}$, then $\mathcal{F}_C^{-2}(y)$ is a closed cone with vertex at $\theta$ in $X$. In general $\mathcal{F}_C^{-2}(y)$ is not convex.
\end{description}
\end{pr}
\begin{proof} We only prove that $\mathcal{F}_C^{-2}(y)$ is not convex. Let $X=\mathds{R}^3$ be as in Example~\ref{KL4-Ex2.1}. Let $y=(25,37,77)$ and define
$C=[\theta,y]=\{\alpha y:\ 0\leq \alpha \leq 1\}$. We take $x=(3,-2,-1)$ and $z=(1,-3,2)$. Then $\|x\|_3=\|z\|_3=\sqrt[3]{36}.$ As before, we compute
$$\langle Jx,y-\alpha y\rangle=0,\quad \text{for every}\ \alpha y\in C,\ \alpha\in [0,1].$$
Therefore, $x\in \mathcal{F}_C^{-2}(y)$. Analogously, we have $z\in \mathcal{F}_C^{-2}(y)$. We take $h=\frac23 x+\frac13 z=\left(\frac73,-\frac73,0\right).$ Proceeding as before, we have
$$\langle Jh,y-\alpha y\rangle<0,\quad \text{for every}\ \alpha y\in C,\ \alpha\in [0,1),$$
proving that $h\notin \mathcal{F}_C^{-2}(y).$ Thus, $\mathcal{F}_C^{-2}(y)$ is not convex which proves the assertion.
\end{proof}

\begin{defn}\label{KKL4-D5.14} Let $X$ be a Banach space with dual $X^*$ and let $C$ be a nonempty, closed, and convex set in $X$. For any $y\in C,$ we define
\begin{description}
\item[(a)] If $\mathcal{F}_C^{-1}(y)=\{\theta^*\}$, then $y$ is called an internal point of $C$.
\item[(b)] If $\mathcal{F}_C^{-1}(y)\supsetneqq \{\theta^*\}$, then $y$ is called a cuticle point of $C$.
\end{description}
The collection of all internal points of $C$ is denoted by $\mathcal{J}(C) $and the collection of all cuticle points of $C$ is denoted by $\mathcal{C}(C)$
\end{defn}

As a direct consequence of Proposition~\ref{KKL4-P5.12}, we have the following result.
\begin{co}\label{KKL4-C5.15} Let $X$ be a Banach space with dual $X^*$ and let $C$ be a nonempty, closed, and convex set in $X$. Then $C=\{\mathcal{J}(C),\mathcal{C}(C)\}$ is a partition of $C$. More precisely, we have $C=\mathcal{J}(C)\cup \mathcal{C}(C).$
\end{co}

\begin{co}\label{KKL4-C5.16} Let $X$ be a uniformly convex and uniformly smooth Banach space with dual $X^*$ and let $C$ be a nonempty, closed, and convex set in $X$. For any $y\in C,$ we have
\begin{description}
\item[(a)] $y\in \mathcal{J}(C)$ if and only if for $\psi\in X^*$, $y=\pi_C(\psi+Jy)$ implies that $\psi=\theta^*.$
\item[(b)] $y\in \mathcal{C}(C)$ if and only if there is $\theta^*\ne \psi\in X^*$ such that $y=\pi_C(\psi+Jy)$.
\end{description}
\end{co}

An analogue of the above result can be given by using the metric projection $P_C$.
\begin{co}\label{KKL4-C5.17} Let $X$ be a uniformly convex and uniformly smooth Banach space with dual $X^*$ and let $C$ be a nonempty, closed, and convex set in $X$. For any $y\in C,$ we have
\begin{description}
\item[(a)] $y\in \mathcal{J}(C)$ if and only if for $\psi\in X^*$, $y=P_C(J^*\psi+y)$ implies that $\psi=\theta^*.$
\item[(b)] $y\in \mathcal{C}(C)$ if and only if there is $\theta^*\ne \psi\in X^*$ such that $y=P_C(J^*\psi+y)$.
\end{description}
\end{co}

Next we give some examples to demonstrate the concepts of $\mathcal{J}(C)$ and $\mathcal{C}(C).$
\begin{co}\label{KKL4-C5.18} Let $X$ be Banach space with dual $X^*$ and let $C$ be a proper closed subspace of $X$. Then:
\begin{description}
\item[(a)] $\mathcal{J}(C)=\emptyset.$
\item[(a)] $\mathcal{C}(C)=C.$
\end{description}
\end{co}
\begin{proof} Since $C$ is a proper closed subspace of $X$, by the Hahn-Banach space theorem, there is $\psi\in X^*$ with $\|\psi\|_{X^*}=1$ such that $\langle \psi,x\rangle=0 $ for all $x\in C.$ This implies that $\psi\in \mathcal{F}_C^{-1}(y)$ for all $y\in C.$ Since $\psi\ne \theta^*$, it follows at once that $y\in \mathcal{C}(C)$, for all $y\in C$. The claim then follows from Corollary~\ref{KKL4-C5.15}.
\end{proof}

In the following result, we use the closed and open balls and the unit sphere, see Section~\ref{KL4-S2-P}. 
\begin{pr}\label{KKL4-P5.19} Let $X$ be Banach space with dual $X^*$. For $r>0$, we have
\begin{description}
\item[(a)] $\mathcal{J}(B(r))=B^0(r).$
\item[(b)] $\mathcal{C}(B(r))=S(r).$
\item[(c)] For any $y\in S(r)$, $\mathcal{F}_{B(r)}^{-1}(y)$ is a closed and convex cone with vertex at $\theta^*$ and
$$\mathcal{F}_{B(r)}^{-1}(y)=\cup_{\psi\in Jy}[\theta^*,\psi\lceil.$$
\item[(d)] If $X^*$ is strictly convex, then for any $y\in S(r)$, we have
$$\mathcal{F}_{B(r)}^{-1}(y)=[\theta^*,Jy\lceil.$$
\end{description}
\end{pr}
\begin{proof} (a) We first prove that $\theta \in \mathcal{J}(B(r))$. For any $\psi\in X^*$ with $\|\psi\|_{X^*}>0$, there is $x\in B(r)$ such that $\langle \psi,x\rangle\ne 0.$ Since for $x\in B(r)$, we also have $-x\in B(r)$, it follows that one of $\langle \psi, x\rangle $ and $\langle \psi, -x\rangle $ is positive. By $\langle \psi,\theta\rangle=0, $ it follows that
$$\psi\notin \mathcal{F}_{B(r)}^{-1}(\theta),\quad \text{for any}\ \psi\in X^*\ \text{with}\ \|\psi\|_{X^*}\ne \emptyset.$$
This implies $\mathcal{F}_{B(r)}^{-1}(\theta)=\{\theta^*\}$, and therefore $\theta \in \mathcal{J}(B(r)).$

For any $y\in B^0(r)$ with $0<\|y\|_X<r$, the proof of $y\in \mathcal{J}(B(r))$ is divided into two parts.\\
\textbf{Case 1.} $\psi\in X^*$ with $\|\psi\|_{X^*}>0$ satisfying $\langle \psi,y\rangle=0.$ Then, there is $z\in B(r)$ and $-z\in B(r)$ such that $\langle \psi,z\rangle\ne 0.$ It follows that one of $\langle \psi, x\rangle $ and $\langle \psi, -x\rangle $ is positive. Then,
\begin{equation}\label{KKL4-E5.9}
\psi\notin \mathcal{F}_{B(r)}^{-1}(y)\quad \text{for any}\ \psi\in X^*\ \text{with}\ \|\psi\|_{X^*}\ne 0\ \text{any}\ \langle \psi,y\rangle=0.
\end{equation}
\textbf{Case 2.} $\psi\in X^*$ with $\|\psi\|_{X^*}>0$ satisfying $\langle \psi,y\rangle\ne 0.$ By $0<\|y\|_X<r$, there are positive numbers $s$ and $t$ with $t>1>s>0$ such that $\|ty\|_X<r$. Then, $ty,ts\in B^0(r)\subset B(r)$. We have
$$\max\{\langle \psi,ty\rangle,\langle \psi,sy\rangle \}>\langle \psi,y\rangle.$$
By $ty,ts\in B(r)$, we deduce that
\begin{equation}\label{KKL4-E5.10}
\psi\notin \mathcal{F}_{B(r)}^{-1}(y),\quad \text{for any}\ \psi\in X^*\ \text{with}\ \|\psi\|_{X^*}\ne 0,\ \text{and}\ \langle \psi,y\rangle\ne 0.
\end{equation}
Combining \eqref{KKL4-E5.9} and \eqref{KKL4-E5.10}, for any $y\in B^0(r)$ with $0<\|y\|_X<r$, we have
$$\psi\notin \mathcal{F}_{B(r)}^{-1}(y),\quad \text{for any}\ \psi\in X^*\ \text{with}\ \|\psi\|_{X^*}\ne 0.$$
which implies that
$$y\in \mathcal{J}(B(r)),\quad \text{for any}\ y\in B^0(r)\ \text{with}\ \|y\|_X>0,$$
which, when combined with the containment $\theta\in  \mathcal{J}(B(r))$ proves (a).

(b) For any $y\in S(r)$ and for any $\psi \in Jy\subseteq X^*$, we have  $\|\psi\|_{X^*}=\|y\|_X=r$ and $\langle \psi ,y\rangle=r^2$. Then,
$$\langle \psi,x\rangle\leq \|\psi\|_{X^*}\|x\|_X\leq r^2=\langle \psi,y\rangle,\quad \text{for any }\ x\in B(r),$$
which implies
\begin{equation}\label{KKL4-E5.12}
\psi\in \mathcal{F}_{B(r)}^{-1}(y),\quad \text{for any}\ y\in S(r).
\end{equation}
Since $\psi\in X^*$ with $\psi\ne \theta $, we have $y\in \mathcal{C}(B(r))$, for any $y\in S(r)$. This, taking into account Remark~\ref{KKL4-L5.3} , implies that $\mathcal{C}(B(r))=S(r).$

(c). From \eqref{KKL4-E5.12}, we have
\begin{equation}\label{KKL4-E5.13}
Jy=\{\psi\in Jy\}\subseteq \mathcal{F}_{B(r)}^{-1}(y),\quad \text{for any}\ y\in S(R).
\end{equation}
Next we show that for any fixed $y\in S(r),$ we have
\begin{equation}\label{KKL4-E5.14}
[\theta^*,\psi \lceil \subseteq \mathcal{F}_{B(r)}^{-1}(y),\quad \text{for any}\ \psi\in Jy.
\end{equation}
From \eqref{KKL4-E5.12}, for any $t\geq 0,$ we have
\begin{equation}\label{KKL4-E5.15}
\langle t\psi, y\rangle=t\langle \psi,y\rangle \geq t\langle \psi,x\rangle=\langle t\psi,x\rangle,\quad \text{for all}\ x\in B(r),
\end{equation}
which implies that $t\psi\in \mathcal{F}_{B(r)}^{-1}(y)$, for any $t\geq 0$, which proves \eqref{KKL4-E5.14}. Therefore,
\begin{equation}\label{KKL4-E5.16}
\cup_{\psi\in Jy}[\theta^*,\psi\lceil \subseteq \mathcal{F}_{B(r)}^{-1}(y),\quad \text{for any}\ y\in S(r).
\end{equation}

On the other hand, for $y\in S(r)$ and for given $\psi\in \mathcal{F}_{B(r)}^{-1}(y)$ with $\psi\ne \theta^*$, as in the proof of \eqref{KKL4-E5.14}, we can show that
\begin{equation}\label{KKL4-E5.17}
[\theta^*,\psi \lceil \subseteq \mathcal{F}_{B(r)}^{-1}(y),\quad \text{for any}\ \psi\in \mathcal{F}_{B(r)}^{-1}(y)\ \text{with}\ \psi\ne \theta^*.
\end{equation}
So, we may assume that $\|\psi\|_{X^*}=r$. It follows that $\|J^*\psi\|_X=r$, which implies $J^*\psi\in S(r)$. By $y\in S(r)$, $\psi\in \mathcal{F}_{B(r)}^{-1}(y)$ and $J^*\psi\in S(r)$, it follows that
$$r^2\geq \|\psi\|_{X^*}\|y\|_X\geq \langle \psi,y\rangle\geq \langle \psi,J^*\psi\rangle=\|\psi\|^2_{X^*} =r^2.$$
This implies $\|\psi\|_{X^*}=\|y\|_X=r$ and $\langle \psi,y\rangle=r^2 $ Hence $\psi\in Jy.$ We have established,
$$\psi\in [\theta^*,\psi\lceil \subseteq \cup_{\psi\in Jy}[\theta^*,\psi\lceil,$$
 implying
 \begin{equation}\label{KKL4-E5.18}
\mathcal{F}_{B(r)}^{-1}(y)\subseteq \cup_{\psi\in Jy}[\theta^*,\psi\lceil,\quad \text{for any}\ y\in S(r).
\end{equation}
By combining \eqref{KKL4-E5.16} and \eqref{KKL4-E5.18}, we complete the proof of (c).

(d) It follows at once from (c) under the additional hypothesis on $X$.
\end{proof}

The following result connects generalized dual cones with the notion of visions.
\begin{thr}\label{KKL4-T5.20} Let $X$ be a uniformly convex and uniformly smooth Banach space and let $K$ be a closed and convex cone in $X$ with vertex at $v$. Then,
\begin{equation}\label{KKL4-E5.19}
K_{\pi}^{\perp}=\pi_K^{-1}(v)=Jv+\mathcal{F}_K^{-1}(v).
\end{equation}
\end{thr}
\begin{proof} By Theorem~\ref{KKL4-T4.3}, we have  $K_{\pi}^{\perp}=\pi_{K}^{-1}(v)$. Thus, we only need to prove the second equality in \eqref{KKL4-E5.19}.  For any $\psi\in X^*$, we have
\begin{align*}
\psi\in \mathcal{F}_K^{-1}(v)\ &\Leftrightarrow\ \langle \psi,v-x\rangle\geq 0,\ \text{for all}\ x\in K,\\
&\Leftrightarrow\ \langle \psi+Jv-Jv,v-x\rangle\geq 0,\ \text{for all}\ x\in K\\
&\Leftrightarrow\ \psi+Jv\in K_{\pi}^{\perp},
\end{align*}
and the proof is complete.
\end{proof}
\begin{rem} Equation \eqref{KKL4-E5.19} reexamines the following results:
\begin{description}
\item[(i)] $K_{\pi}^{\perp}$ is a closed and convex cone with vertex at $Jv$ in $X^*$ (Theorem~\ref{KKL4-T4.3}).
\item[(ii)] $\mathcal{F}_K^{-1}(v)$ is a closed and convex cone with vertex at $\theta^*$ in $X^*$ (Proposition~\ref{KKL4-P5.13}).
\end{description}
\end{rem}

%\begin{acknowledgements}
%  The research of Akhtar Khan is supported by the National Science Foundation under Award No. 1720067.  Miguel Sama is supported by Ministerio de Ciencia, Innovaci\'{o}n y Universidades (MCIU), Agencia Estatal de Investigaci\'{o}n (AEI) %(Spain) and Fondo Europeo de Desarrollo Regional (FEDER) under project PGC2018-096899-B-I00 (MCIU/AEI/FEDER, UE) and grant number
%2020-MAT11 (ETSI Industriales, UNED).
%\end{acknowledgements}
\bibliographystyle{spmpsci_unsrt}
\bibliography{BIB-PM}

\begin{thebibliography}{10}
\providecommand{\url}[1]{{#1}}
\providecommand{\urlprefix}{URL }
\expandafter\ifx\csname urlstyle\endcsname\relax
  \providecommand{\doi}[1]{DOI~\discretionary{}{}{}#1}\else
  \providecommand{\doi}{DOI~\discretionary{}{}{}\begingroup
  \urlstyle{rm}\Url}\fi

\bibitem{Zar71}
Zarantonello, E.H.: Projections on convex sets in {H}ilbert space and spectral
  theory. {I}. {P}rojections on convex sets pp. 237--341 (1971)

\bibitem{BalGol12}
Balashov, M.V., Golubev, M.O.: About the {L}ipschitz property of the metric
  projection in the {H}ilbert space.
\newblock J. Math. Anal. Appl. \textbf{394}(2), 545--551 (2012)

\bibitem{BalMarTei21}
Balestro, V., Martini, H., Teixeira, R.: Convex analysis in normed spaces and
  metric projections onto convex bodies.
\newblock J. Convex Anal. \textbf{28}(4), 1223--1248 (2021)

\bibitem{Bau03}
Bauschke, H.H.: The composition of projections onto closed convex sets in
  {H}ilbert space is asymptotically regular.
\newblock Proc. Amer. Math. Soc. \textbf{131}(1), 141--146 (2003)

\bibitem{BorDruChe17}
Borodin, P.A., Druzhinin, Y.Y., Chesnokova, K.V.: Finite-dimensional subspaces
  of {$L_p$} with {L}ipschitz metric projection.
\newblock Mat. Zametki \textbf{102}(4), 514--525 (2017)

\bibitem{Bou15}
Bounkhel, M.: Generalized projections on closed nonconvex sets in uniformly
  convex and uniformly smooth {B}anach spaces.
\newblock J. Funct. Spaces pp. Art. ID 478,437, 7 (2015)

\bibitem{Bro13}
Brown, A.L.: On lower semi-continuous metric projections onto finite
  dimensional subspaces of spaces of continuous functions.
\newblock J. Approx. Theory \textbf{166}, 85--105 (2013)

\bibitem{BroDeu72}
Brosowski, B., Deutsch, F.: Some new continuity concepts for metric
  projections.
\newblock Bull. Amer. Math. Soc. \textbf{78}, 974--978 (1972)

\bibitem{Bui02}
Kien, B.T.: On the metric projection onto a family of closed convex sets in a
  uniformly convex {B}anach space.
\newblock Nonlinear Anal. Forum \textbf{7}(1), 93--102 (2002)

\bibitem{Bur21}
Burusheva, L.S.: An example of a {B}anach space with non-lipschitzian metric
  projection on any straight line.
\newblock Mat. Zametki \textbf{109}(2), 196--205 (2021)

\bibitem{CheGol59}
Cheney, W., Goldstein, A.A.: Proximity maps for convex sets.
\newblock Proc. Amer. Math. Soc. \textbf{10}, 448--450 (1959)

\bibitem{ChiLi05}
Chidume, C.E., Li, J.L.: Projection methods for approximating fixed points of
  {L}ipschitz suppressive operators.
\newblock PanAmer. Math. J. \textbf{15}(1), 29--39 (2005)

\bibitem{Den01}
Dentcheva, D.: On differentiability of metric projections onto moving convex
  sets.
\newblock pp. 283--298 (2001).
\newblock Optimization with data perturbations, II

\bibitem{DeuLam80}
Deutsch, F., Lambert, J.M.: On continuity of metric projections.
\newblock J. Approx. Theory \textbf{29}(2), 116--131 (1980)

\bibitem{DutShuTho17}
Dutta, S., Shunmugaraj, P., Thota, V.: Uniform strong proximinality and
  continuity of metric projection.
\newblock J. Convex Anal. \textbf{24}(4), 1263--1279 (2017)

\bibitem{GJKS21}
Gwinner, J., Jadamba, B., Khan, A.A., Raciti, F.: Uncertainty Quantification in
  Variational Inequalities.
\newblock CRC Press (2021)

\bibitem{Ind14}
Indumathi, V.: Semi-continuity properties of metric projections.
\newblock In: Nonlinear analysis, Trends Math., pp. 33--59.
  Birkh\"{a}user/Springer, New Delhi (2014)

\bibitem{FitPhe82}
Fitzpatrick, S., Phelps, R.R.: Differentiability of the metric projection in
  {H}ilbert space.
\newblock Trans. Amer. Math. Soc. \textbf{270}(2), 483--501 (1982)

\bibitem{KonLiuLiWu22}
Kong, D., Liu, L., Li, J., Wu, Y.: Isotonicity of the metric projection with
  respect to the mutually dual orders and complementarity problems.
\newblock Optimization \textbf{71}(16), 4855--4877 (2022)

\bibitem{KroPin13}
Kro\'{o}, A., Pinkus, A.: On stability of the metric projection operator.
\newblock SIAM J. Math. Anal. \textbf{45}(2), 639--661 (2013)

\bibitem{Li04}
Li, J.L.: The metric projection and its applications to solving variational
  inequalities in {B}anach spaces.
\newblock Fixed Point Theory \textbf{5}(2), 285--298 (2004)

\bibitem{Li04a}
Li, J.L.: On the existence of solutions of variational inequalities in {B}anach
  spaces.
\newblock J. Math. Anal. Appl. \textbf{295}(1), 115--126 (2004)

\bibitem{LiZhaMa08}
Li, J.L., Zhang, C., Ma, X.: On the metric projection operator and its
  applications to solving variational inequalities in {B}anach spaces.
\newblock Numer. Funct. Anal. Optim. \textbf{29}(3-4), 410--418 (2008)

\bibitem{Nak22}
Nakajo, K.: Strong convergence for the problem of image recovery by the metric
  projections in {B}anach spaces.
\newblock J. Nonlinear Convex Anal. \textbf{23}(2), 357--376 (2022)

\bibitem{Osh70}
O\v{s}man, E.V.: \v{C}eby\v{s}ev sets and the continuity of metric projection.
\newblock Izv. Vys\v{s}. U\v{c}ebn. Zaved. Matematika \textbf{1970}(9 (100)),
  78--82 (1970)

\bibitem{Pen05}
Penot, J.P.: Continuity properties of projection operators.
\newblock J. Inequal. Appl. (5), 509--521 (2005)

\bibitem{PenRat98}
Penot, J.P., Ratsimahalo, R.: Characterizations of metric projections in
  {B}anach spaces and applications.
\newblock Abstr. Appl. Anal. \textbf{3}(1-2), 85--103 (1998)

\bibitem{QiuWan22}
Qiu, Y., Wang, Z.: The metric projections onto closed convex cones in a
  {H}ilbert space.
\newblock J. Inst. Math. Jussieu \textbf{21}(5), 1617--1650 (2022)

\bibitem{Ric16}
Ricceri, B.: More on the metric projection onto a closed convex set in a
  {H}ilbert space.
\newblock In: Contributions in mathematics and engineering, pp. 529--534.
  Springer, Cham (2016)

\bibitem{Sha16}
Shapiro, A.: Differentiability properties of metric projections onto convex
  sets.
\newblock J. Optim. Theory Appl. \textbf{169}(3), 953--964 (2016)

\bibitem{ShaZha17}
Shang, S., Zhang, J.: Metric projection operator and continuity of the
  set-valued metric generalized inverse in {B}anach spaces.
\newblock J. Funct. Spaces pp. Art. ID 7151,430, 8 (2017)

\bibitem{ZhaZhoLiu19}
Zhang, Z., Zhou, Y., Liu, C.: Continuity of generalized metric projections in
  {B}anach spaces.
\newblock Rev. R. Acad. Cienc. Exactas F\'{\i}s. Nat. Ser. A Mat. RACSAM
  \textbf{113}(1), 95--102 (2019)

\bibitem{Alb93}
Alber, Y.I.: Generalized projection operators in {B}anach spaces: properties
  and applications.
\newblock In: Functional-differential equations, \emph{Funct. Differential
  Equations Israel Sem.}, vol.~1, pp. 1--21. Coll. Judea Samaria, Ariel (1993)

\bibitem{Alb96}
Alber, Y.I.: Metric and generalized projection operators in {B}anach spaces:
  properties and applications.
\newblock In: Theory and applications of nonlinear operators of accretive and
  monotone type, \emph{Lecture Notes in Pure and Appl. Math.}, vol. 178, pp.
  15--50. Dekker, New York (1996)

\bibitem{KhaLiRwi22}
Khan, A.A., Li, J.L., Reich, S.: Generalized projection operators on general
  banach spaces.
\newblock Journal of Nonlinear and Convex Analysis (at press)  (2023)

\bibitem{Li05}
Li, J.L.: The generalized projection operator on reflexive {B}anach spaces and
  its applications.
\newblock J. Math. Anal. Appl. \textbf{306}(1), 55--71 (2005)

\bibitem{Tak00}
Takahashi, W.: Nonlinear functional analysis. Fixed point theory and its
  applications.
\newblock Yokohama Publishers, Yokohama (2000)

\bibitem{KhaLi23}
Khan, A.A., Li, J.L.: Approximating properties of metric and generalized metric
  projections in uniformly convex and uniformly smooth banach spaces.
\newblock Under review pp. 1--14 (2023)

\end{thebibliography}
\end{document}